\documentclass[11pt, reqno]{amsart}
\usepackage{amsfonts}
\usepackage{hyperref}
\usepackage{graphicx}
\usepackage{tabularx}
\usepackage{array}
\usepackage[usenames,dvipsnames]{color}
\usepackage[T1]{fontenc}
\usepackage[utf8]{inputenc}
\usepackage[russian, french, english]{babel}
\usepackage{comment}
\usepackage{amsmath}
\usepackage{amsthm}
\usepackage{amssymb}
\usepackage{todonotes}
\usepackage{fullpage}
\usepackage{xcolor}
\usepackage{tikz}
\usepackage{listings}
\usepackage{soul}
\usepackage[noindentafter]{titlesec}
\usepackage{mathtools}
\usepackage{enumerate, enumitem}

\tikzstyle{legend_general}=[rectangle, rounded corners, thin,
                          top color= white,bottom color=lavander!25,
                          minimum width=2.5cm, minimum height=0.8cm,
                          violet]
\hypersetup{
    colorlinks=true,
    urlcolor=blue,
    linkcolor=blue,
    citecolor=OliveGreen,
    linktoc=page,
}

\DeclareMathOperator{\mad}{mad}

\renewcommand{\epsilon}{\varepsilon}

\newtheorem{theorem}{Theorem}[section]
\newtheorem*{theorem*}{Theorem}

\newtheorem{lemma}[theorem]{Lemma}
\newtheorem*{lemma*}{Lemma}

\theoremstyle{definition}

\newtheorem{claim}[theorem]{Claim}

\titleformat{\section}[block]{\scshape\filcenter}{\thesection.}{1ex}{}
\titleformat{\subsection}[block]{\bfseries}{\thesubsection.}{1ex}{}
\titleformat{\subsubsection}[runin]{\itshape}{\bfseries\upshape\thesubsubsection.}{1ex}{}[.---]
\titleformat{\paragraph}[runin]{\normalfont\normalsize\bfseries}{}{1em}{}
\titlespacing*{\paragraph}{0pt}{3.25ex plus 1ex minus .2ex}{1em}

\usepackage{algorithm}
\usepackage{algpseudocode}
\counterwithin{algorithm}{section}



\title{Density of list- and correspondence-critical graphs}

\thanks{Peter Bradshaw received funding from NSF RTG grant DMS-1937241.}

\author{Peter Bradshaw}
\address{Department of Mathematics, University of Illinois Urbana-Champaign}
\email{bradshap@mailbox.sc.edu}

\begin{document}
\maketitle
\begin{abstract}
    A graph $G$ is list $k$-critical if $G$ is not  $(k-1)$-list-colorable, but every proper subgraph of $G$ is $(k-1)$-list-colorable.
    In this paper, we study the function $f_{\ell}(n,k)$ denoting
    the minimum number of edges in an $n$-vertex list $k$-critical graph, 
    as well as the function
    $g_{\ell}(k) = \liminf_{n \rightarrow \infty} \frac 2n f_{\ell}(n,k) - k + 1$.
    We show that for all $k \geq 4$, 
    every list
    $k$-critical graph on $n \geq k+2$ vertices has more than $(k-1+\frac 1{28}) \frac n2$ edges, which implies that $g_{\ell}(k)
    \geq \frac{1}{28}$ for all $k \geq 4$.
     This is the first result showing that $\liminf_{k \rightarrow \infty} g_{\ell}(k) > 0$.
    We also show that $g_{\ell}(k) \geq \frac{1}{24}$ for all $k \geq 352$.
    As a corollary to our result, we obtain the following improvement to Brooks' theorem:
    For all $d \geq 3$, if $G$ has no $K_{d+1}$ subgraph and has maximum average degree at most $d+\frac 1{28}$, then $G$ is $d$-list-colorable.

    All of our results hold in the setting of correspondence coloring (DP-coloring) as well. 
    As a corollary of our correspondence coloring result, we also show that for each $d \geq 3$,
    a minimal unsatisfiable anti-functional binary constraint satisfaction problem (CSP) with variable domains of size $d$ has a primal graph 
    whose unique nontrivial component
    either contains
    $K_{d+1}$ or has  average degree at least $d+\frac 1{28}$.
\end{abstract}
\section{Introduction}
We consider \emph{graphs}, which we define as having no loops or parallel edges.
We also consider \emph{multigraphs},
which possibly have parallel edges but no loops.


\subsection{Background: Proper coloring}
A \emph{proper coloring} of a graph $G$ is an assignment $\phi:V(G) \rightarrow \mathbb N$ such that for every adjacent vertex pair $u,v \in V(G)$, $\phi(u) \neq \phi(v)$.
A graph $G$ is \emph{$k$-colorable} if $G$ has a proper coloring $\phi:V(G) \rightarrow \{1, \dots, k\}$, and the function $\phi$ is called a \emph{$k$-coloring} of $G$.
If $G$ has no $(k-1)$-coloring but every proper subgraph of $G$ has a $(k-1)$-coloring, then $G$ is \emph{$k$-critical}.

In 1951, Dirac~\cite{1951Dirac}
introduced the notion of a critical graph, 
and in a 
series of subsequent papers \cite{1952DiracAproperty,
 1953Dirac,1957DiracMap,1957DiracAtheorem}, he studied
 the minimum number $f(n,k)$ of edges in a $k$-critical graph on $n$ vertices.
As $K_k$ is the only $k$-critical graph on $k$ vertices, it follows that $f(k,k) = \binom k2$, and as no $k$-critical graph on $k+1$ vertices exists,
$f(k+1,k)$ is undefined.
Every $k$-critical $n$-vertex graph with $n \geq k+2$
has minimum degree at least $k-1$ and a vertex of degree at least $k$ by Brooks' Theorem,
which implies that $f(n,k) \geq \frac 12 ((k-1) n+1)$ for all $k \geq 4$ and $n \geq k+2$.
In 1957, Dirac established the following improved lower bound for $f(n,k)$:
\begin{theorem}[\cite{1957DiracAtheorem}]
\label{thm:Dirac}
    For all $k \geq 4$ and $n \geq k+2$, $f(n,k) \geq  (k-1) \frac n2 + \frac 12 (k-3)$.
\end{theorem}
Later, Dirac \cite{1974Dirac}
also characterized the $k$-critical graphs for which the bound in Theorem \ref{thm:Dirac} is exact.

When $k$ is fixed and $n$ increases,
the lower bound in Theorem \ref{thm:Dirac} is asymptotically equivalent to the Brooks lower bound of $f(n,k) \geq \frac 12 ((k-1) n+1)$.
With the goal of finding an improved lower bound in this setting, we define the function $g(k) = \liminf_{n \rightarrow \infty} \frac 2n f(n,k) - k + 1$, and we observe that the Brooks lower bound and Theorem \ref{thm:Dirac} only imply that $g(k) \geq 0$ for all $k \geq 4$.
The first improved lower bound for $g(k)$ came from Gallai, who showed that $g(k) \geq \frac{k-3}{k^2-3}$ for all $k \geq 4$:
\begin{theorem}[\cite{1963Gallai1}]
\label{thm:Gallai-fnk}
    For all $k \geq 4$ and $n \geq k+2$, $f(n,k) > \left ( k-1+ \frac{k-3}{k^2-3} \right ) \frac n2$.
\end{theorem}
The main ingredient in Gallai's improved lower bound is a characterization of the structure of the vertices of degree $k-1$ in a $k$-critical graph. 
We say that a connected graph $T$ is a \emph{Gallai tree}
if every block of $T$ is a clique or an odd cycle.
We say that a \emph{Gallai forest} is a graph in which every component is a Gallai tree.
Theorem \ref{thm:Gallai-fnk} follows from a combination of the following observation of Gallai and a counting argument:
\begin{theorem}[\cite{1963Gallai2}]
\label{thm:Gallai-L}
    Let $G$ be a $k$-critical graph, and let $\mathcal L$ be the subgraph of $G$ induced by the vertices of degree $k-1$. Then, $\mathcal L$ is a Gallai forest.
\end{theorem}
The bound $g(k) \geq \frac{k-3}{k^2 - 3}$ in Theorem \ref{thm:Gallai-fnk}
was improved several times.
First, Krivelevich~\cite{1997Krivelevich,1998Krivelevich}
showed that $g(k) \geq \frac{k-3}{k^2 - 2k - 1}$ for $k\geq 4$,  and then Kostochka and Stiebitz~\cite{2003KS} showed that $g(k) \geq \frac{2(k-3)}{k^2+6k-9-6/(k-2)}$ for $k \geq 6$.
In 2014, Kostochka and Yancey
finally
showed that $g(k) = \frac{k-3}{k-1}$ for all $k \geq 4$, which is best possible due to a construction of Haj\'os \cite{1961Hajos}:
\begin{theorem}[\cite{2014KoYa}]
\label{thm:KY}
    Let $k \geq 4$, and let $n \geq k+2$. Then,
    $f(n,k) \geq \left ( k-1+\frac{k-3}{k-1} \right ) \frac n2 - \frac{k(k-3)}{2(k-1)}$.
\end{theorem}

Theorem \ref{thm:KY}
is a remarkable breakthrough,
as it shows that $\lim_{k \rightarrow \infty} g(k) = 1$,
whereas all previous results fail even to show that $\liminf_{k \rightarrow \infty} g(k) > 0$.

\subsection{Background: List coloring and correspondence coloring}
A \emph{list assignment} on a graph $G$ is a function $L:V(G) \rightarrow 2^{\mathbb N}$
that assigns a \emph{list} $L(v) \subseteq \mathbb N$ to each vertex $v \in V(G)$.
Then, an \emph{$L$-coloring} of $G$ is a proper coloring $\phi:V(G) \rightarrow \mathbb N$ satisfying $\phi(v) \in L(v)$ 
for every $v \in V(G)$. 
If $G$ has an $L$-coloring for every list assignment $L$ satisfying
$|L(v)| \geq k$ for each $v\in V(G)$, then $G$ is \emph{$k$-list-colorable}. 
If $G$ is not $(k-1)$-list-colorable but every proper subgraph of $G$ is $(k-1)$-list-colorable, then $G$ is \emph{list $k$-critical}.
Note that if $G$ is $k$-list-colorable, then $G$ has an $L$-coloring in particular for the list assignment $L$ that assigns $L(v) = \{1, \dots, k\}$ to each vertex $v$, 
so $G$ is also $k$-colorable.

We write $f_{\ell}(n,k)$ for the minimum number of edges in a list $k$-critical graph on $n$ vertices, and we also define $g_{\ell}(k) = \liminf_{n \rightarrow \infty} \frac 2n  f_{\ell}(n,k) - k + 1$.
Kostochka and Stiebitz \cite{2003KS} initiated the study of the function $f_{\ell}(n,k)$ in 2003;
however, earlier results in the setting of list coloring yield immediate bounds for this function.
One such result is Vizing's extension of Brooks' Theorem to list coloring \cite[Theorem 2]{1976Vizing},
which implies that $f_{\ell}(n,k) \geq \frac 12 ( (k-1) n + 1)$ for all $k \geq 4$ and  $n \geq k+2$.
Another such result concerns \emph{degree choosability}, which is defined as follows.
Given a graph $G$, a \emph{degree list assignment} is a list assignment $L:V(G) \rightarrow 2^{\mathbb N}$ such that $|L(v)| \geq d(v)$ for every $v \in V(G)$.
The following result
found by Borodin \cite{1979Borodin} 
and independently by Erd\H os, Rubin, and Taylor \cite{1980ErRuTa} 
is a generalization of Theorem \ref{thm:Gallai-L}.

\begin{theorem}
\label{thm:deg-choosable}
    Let $G$ be a graph, and let $L$ be a degree list assignment.
    If $G$ has no $L$-coloring, then $|L(v)| = d(v)$ for each $v \in V(G)$, and $G$ is a Gallai forest.
\end{theorem}
Theorem \ref{thm:deg-choosable} implies that the lower bound in Theorem \ref{thm:Gallai-fnk} also holds for $f_{\ell}(n,k)$ and hence that $g_{\ell}(k) \geq \frac{k-3}{k^2-3}$ for all $k \geq 4$.

Several improved lower bounds have been established for the functions $f_{\ell}(n,k)$
and $g_{\ell}(k)$,
including bounds of Kostochka and Stiebitz~\cite{2003KS}, Kierstead and Rabern~\cite{2020KiRa}, Cranston and Rabern~\cite{2018CrRa}, and Rabern~\cite{2016Rabern,2018Rabern}.
Currently, the best known lower bound for $f_{\ell}(n,k)$
for fixed $k$ and large $n$ is due to Choi, Kostochka, Xu, and the author:
\begin{theorem}[\cite{BCKX2,BCKX}]
\label{thm:BCKX}
    Let $k \geq 4$.
    Let $n \geq 11$ when $k = 4$, and let $n \geq k+2$ when $k \geq 5$. Then,
    \[
        f_{\ell}(n,k) \geq \begin{cases}
            \frac 85 n + \frac 15 & \text{ if } k = 4, \\
            \left ( k - 1 + \left \lceil \frac{k^2-7}{2k-7} \right \rceil^{-1} \right ) \frac n2 + \left \lceil \frac{k^2-7}{2k-7} \right \rceil^{-1} & \text{ if } k \geq 5.
        \end{cases}
    \]
\end{theorem}   
Theorem \ref{thm:BCKX}
shows that $g_{\ell}(4) \geq \frac 15$ and that $g_{\ell}(k) \geq \left \lceil \frac{k^2-7}{2k-7} \right \rceil^{-1} $ for all $k \geq 5$.
Thus, in contrast to the function $g(k)$, it is unclear based on current knowledge whether $\lim_{k \rightarrow \infty} g_{\ell}(k)$ exists and whether $\liminf_{k \rightarrow \infty} g_{\ell} (k) > 0$.

One point of interest about Theorem \ref{thm:BCKX}
is that the lower bound also holds for graphs that are critical for \emph{correspondence coloring}, defined as follows.
Given a multigraph $G$, a \emph{correspondence cover} of $G$ is a pair $(H,L)$ satisfying the following properties:
\begin{itemize}
    \item $H$ is a multigraph, and $L:V(G) \rightarrow 2^{V(H)}$ assigns a subset $L(v) \subseteq V(H)$ to each $v \in V(G)$;
    \item For each $v \in V(G)$, the set $L(v) \subseteq V(H)$  is independent;
    \item The family $\{L(v):v \in V(G)\}$ partitions $V(H)$;
    \item For each pair $u,v \in V(G)$, $H[L(u) \cup L(v)]$ is a union of $|E_G(u,v)|$ matchings.
\end{itemize}
We call the elements of $V(H)$ \emph{nodes}.
We note that if $|L(v)| = k$ for each $v \in V(G)$ and each edge $uv \in E(G)$ corresponds to a perfect matching $H[L(u) \cup L(v)]$, 
then $H$ is a $k$-fold topological covering of $G$.

Given a correspondence cover $(H,L)$ of a multigraph $G$,
an \emph{$(H,L)$-coloring} of $G$ 
is a function $\phi:V(G) \rightarrow V(H)$ 
that satisfies $\phi(v) \in L(v)$ for each $v \in V(G)$ and whose image is an independent set.
We often identify the function $\phi$ with its image.
A multigraph $G$ is \emph{correspondence $k$-colorable} if $G$ has an $(H,L)$-coloring for every correspondence cover $(H,L)$ of $G$ satisfying $|L(v)| \geq k$ for each $v \in V(G)$.
If $G$ is not correspondence $(k-1)$-colorable, but every proper subgraph of $G$ is correspondence $(k-1)$-colorable, 
then $G$ is \emph{correspondence $k$-critical}.
Correspondence coloring was introduced by Dvo\v r\'ak and Postle \cite{2018DvPo} in 2018
as a tool for showing that every planar graph with no cycle of length $4$ through $8$ is $3$-list-colorable, and it is often called \emph{DP-coloring}.
We write $f_{DP}(n,k)$
for the minimum number of edges in a correspondence $k$-critical (simple) graph on $n$ vertices, and we define $g_{DP}(k) = \liminf_{n \rightarrow \infty} \frac 2n f_{DP}(n,k) - k + 1$.

As the lower bound in Theorem \ref{thm:BCKX}
also holds for $f_{DP}(n,k)$, it follows that $g_{DP}(4) \geq \frac 15$ and that $g_{DP}(k) \geq \left \lceil \frac{k^2-7}{2k-7} \right \rceil^{-1} $ for each $k \geq 5$.
In fact, a construction in \cite{BCKX2} shows that $g_{DP}(4) = \frac 15$.
As with the function $g_{\ell}(k)$,
 it is unclear based on current knowledge whether $\lim_{k \rightarrow \infty} g_{DP}(k)$ exists and whether $\liminf_{k \rightarrow \infty} g_{DP} (k) > 0$.

In the setting of correspondence coloring, Theorem \ref{thm:deg-choosable} has an important analogue.
Given a multigraph $G$, a \emph{degree cover} of $G$ is a correspondence cover $(H,L)$ that satisfies $|L(v)| \geq d(v)$ for each $v \in V(G)$.
We say that $G$ is \emph{degree correspondence-colorable} if $G$ has an $(H,L)$-coloring for every degree cover $(H,L)$ of $G$.
The multigraphs that are not degree correspondence colorable have a structural characterization similar to Theorem \ref{thm:deg-choosable}.
Given a multigraph $G$ and an integer $r \geq 1$,
the multigraph $rG$ is obtained from $G$ by replacing each edge of $G$ with $r$ parallel edges, and each graph $rG$ is called a \emph{multiple} of $G$.
A \emph{GDP-tree} (named after Gallai, Dvo\v r\'ak, and Postle) is a connected multigraph in which every block is a multiple of a clique or a cycle, and a \emph{GDP-forest} is a multigraph in which each component is a GDP-tree.
The following theorem of Bernshteyn, Kostochka, and Pron'
\cite{2017BeKoPr}
is a correspondence analogue of Theorems \ref{thm:Gallai-L} and \ref{thm:deg-choosable}.
\begin{theorem}[\cite{2017BeKoPr}]
\label{thm:GDP-char}
    If a multigraph $G$ is not degree correspondence-colorable, then $G$ is a GDP-forest.
    Furthermore, if $(H,L)$ is a degree cover of $G$ that admits no $(H,L)$-coloring, then $|L(v)| = d(v)$ for each $v \in V(G)$.
\end{theorem}

\subsection{Our results}
Our main result is the following theorem.
\begin{theorem}
\label{thm:intro}
    Let $G$ be a correspondence $k$-critical (simple) graph that is not $K_k$. 
    \begin{enumerate}
        \item If $k \geq 352$, then $2|E(G)| \geq (k-1+\frac{1}{24} )|V(G)|+ \frac{k-1}{24}$.
        \item If $k \geq 47$, then $2|E(G)| \geq (k-1+\frac{1}{28} )|V(G)|+ \frac{k-1}{28}$.
    \end{enumerate}
\end{theorem}
From Theorem \ref{thm:intro}, we also deduce the following lower bounds on $f_{\ell}(n,k)$ and $f_{DP}(n,k)$.
\begin{theorem}
\label{thm:fnk}
    Let $k \geq 47$ and $n \geq k+2$ be integers. Then,
    \begin{enumerate}
        \item $f_{\ell}(n,k) \geq (k-1+\frac 1{28}) \frac n2 + \frac{k-1}{56}$, and $f_{DP}(n,k) \geq (k-1+\frac 1{28}) \frac n2 + \frac{k-1}{56}$;
        \item If $k \geq 352$,
        then $f_{\ell}(n,k) \geq (k-1+\frac 1{24}) \frac n2 + \frac{k-1}{48}$, and $f_{DP}(n,k) \geq (k-1+\frac 1{24}) \frac n2 + \frac{k-1}{48}$.
    \end{enumerate}
\end{theorem}

\begin{table}[h!]
\centering
\begin{tabular}{|r||c|c|| c | c | c |} 
 \hline
$k$ & \cite{BCKX} & This paper & $k$ & \cite{BCKX} & This paper\\ 
 \hline
4 &  $\mathbf{1/5}$ &- & 75 &  $1/40$ & $\mathbf{1/28}$\\
10 &$\mathbf{1/8}$  & - & 100 &  $1/52$ & $\mathbf{1/28}$ \\
30 &  $\mathbf{1/17}$ & - & 200 & $1/102$ & $\mathbf{1/28}$ \\
50 &  $\mathbf{1/27}$ &$1/28$ & 300 &  $1/152$ & $\mathbf{1/28}$ \\
51 & $ \mathbf{1/28}$ &$\mathbf{1/28}$ & 352 &  $1/178$ & $\mathbf{1/24}$ \\
52 & $ \mathbf{1/28}$ &$\mathbf{1/28}$ & 500 & $1/252$ & $\mathbf{1/24}$ \\
53 &  $ 1/29$ &$\mathbf{1/28}$ & 1000 & $1/502$ & $\mathbf{1/24}$ \\
\hline
\end{tabular}
\caption{Lower bounds for $g_{\ell}(k)$ and $g_{DP}(k)$
for selected values of $k$
established by \cite{BCKX} and this paper. Currently best-known lower bounds are shown in bold.}
\label{table:1}
\end{table}

Together with Theorem \ref{thm:BCKX},
Theorem \ref{thm:fnk}
implies that $g_{\ell}(k)\geq \frac 1{28}$ and  
$g_{DP}(k)\geq \frac 1{28}$
for all $k \geq 4$.
When $k \geq 352$, Theorem \ref{thm:fnk} also implies that $g_{\ell}(k) \geq \frac 1{24}$ and $g_{DP}(k) \geq \frac 1{24}$.
Lower bounds on $g_{DP}(k)$ for some selected values of $k$ are shown in Table \ref{table:1}.
These two theorems also
imply the following improved version of Brooks' Theorem.
Given a graph $G$, we write $\mad(G) = \max_{\emptyset \neq F \subseteq G} \frac{2|E(F)|}{|V(F)|}$ for the maximum average degree of $G$.

\begin{theorem}
Let $d \geq 3$. If $G$ is a (simple) graph with no $K_{d+1}$ subgraph and $\mad(G) \leq d+\frac{1}{28}$, then $G$ is correspondence $d$-colorable. 
Furthermore, if $d \geq 351$ and $\mad (G) \leq d+\frac{1}{24}$, then $G$ is correspondence $d$-colorable.
\end{theorem}
\begin{proof}
    If $G$ is not correspondence $d$-colorable, then $G$ has a correspondence $(d+1)$-critical subgraph $G'$ which is not a clique.
    If $d \geq 351$, then Theorem \ref{thm:fnk}
    implies that $\mad(G) \geq \frac{2|E(G')|}{|V(G')|} \geq  \frac{2f_{DP}(|V(G')|,d+1)}{|V(G')|} > d + \frac{1}{24}$.
    If $50 \leq d \leq 350$, then Theorem \ref{thm:fnk}
    implies that $\mad(G) \geq \frac{2|E(G')|}{|V(G')|} \geq  \frac{2f_{DP}(|V(G')|,d+1)}{|V(G')|} > d + \frac{1}{28}$.
    If $3 \leq d \leq 49$, then Theorem \ref{thm:BCKX} implies that $\mad(G) \geq \frac{2|E(G')|}{|V(G')|} \geq \frac{2f_{DP}(|V(G')|,d+1)}{|V(G')|} > d + \frac{1}{28}$.
    In all cases, we have a contradiction.
\end{proof}

The values of $\frac 1{28}$ and $\frac 1{24}$ in Theorem \ref{thm:intro} are likely not optimal, and we make little effort to optimize these quantities.
It seems likely that $\lim_{k \rightarrow \infty} g_{\ell}(k) = \lim_{k \rightarrow \infty} g_{DP}(k) =1$,
but significant improvements to the value of $\frac 1{24}$
seem
to require more involved case analysis,
and even under the most ideal of circumstances, 
it seems unlikely that the method here can achieve better than $\frac 14$ without new ideas.

\subsection{Our approach}
In order to prove Theorem \ref{thm:intro},
we use the \emph{potential method}, which was popularized by Kostochka and Yancey \cite{2014KoYa}.
With this method, we define a \emph{potential function}, which gives a certain measure of graph density.
In its most basic form,
given a graph $G$,
a potential function assigns a value $\rho(v) = (k-1)\lambda + 1$ to each vertex $v \in V(G)$ and assigns $\rho(e) = -2\lambda$ to each $e \in E(G)$, 
where $k$ and $\lambda \in \{24,28\}$ are taken from the inequality in Theorem \ref{thm:intro} that we wish to prove.
In this way, $\rho(G):=\sum_{x \in V(G) \cup E(G)} \rho(x) = 0$ if and only if $2|E(G)| = (k - 1 + \frac 1{\lambda}) |V(G)|$, and the inequality that we wish to prove is equivalent to $\rho(G) \leq -k+1$.
We will consider a minimum counterexample $G$ for which $\rho(G) > -k+1$, and using the minimality of $G$, we will establish various properties of $G$ in terms of the potential function $\rho$.
We will find that working with the function $\rho$ is often much more convenient than working with edge densities $\frac{2|E(G)|}{|V(G)|}$, as we can easily compute how local changes in $G$ affect the value $\rho(G)$ without needing to count the overall number of vertices or edges in $G$.

Our potential function is closely related to the potential function of \cite{BCKX} but with a few important differences, which are discussed in Section \ref{sec:p-setup}.
In short, the bound of $\rho(G) \leq -k+1$ that we aim to prove for our potential function is much stronger than the corresponding bound of $\rho(G) \leq -2$ proven in \cite{BCKX}.
By using an inductive approach to prove this stronger bound, the stronger inductive hypothesis allows us to 
deduce much stronger properties of our counterexample $G$ than what is possible using the approach in \cite{BCKX},
which ultimately allows us to reach a stronger conclusion.

For technical reasons, we work with multigraphs instead of just simple graphs, and we also consider correspondence covers $(H,L)$ of $G$
for which the values $|L(v)|$ are 
not held constant at $k-1$ but are rather
given by some function $\ell:V(G) \rightarrow \{0,1,\dots,k-1\}$.
When we have vertices $v \in V(G)$ for which $\ell(v)< k-1$ or sets of parallel edges in $G$,
our potential function behaves
differently for those vertices and edges.
We will see that these framework relaxations allow for inductive arguments that would otherwise be difficult to carry out. 
A similar relaxed framework is used in \cite{BCKX}.

\subsection{Preliminaries}
Let $G$ be a multigraph.
Given disjoint vertex sets $U, U' \subseteq V(G)$, we write $E_G(U,U')$ for the set of edges with an endpoint in each of $U,U'$, and we write $\|U, U'\| = |E_G(U,U')|$.
When $U$ or $U'$ is a singleton set $\{v\}$, we often write $v$ instead of $\{v\}$.
The \emph{degree} of a vertex $v \in V(G)$, written $d_G(v)$ or just $d(v)$, is the number of edges in $G$ incident to $v$.
The \emph{neighborhood}
of a vertex $v \in V(G)$, written $N_G(v)$ or just $N(v)$,
is the set of vertices $u \in V(G)$ for which $\|u,v\| \geq 1$.
Note that $|N(v)| = d(v)$ if and only if $\|u,v\| = 1$ for every vertex $u \in N(v)$.

Let $G$ have a correspondence cover $(H,L)$.
Given a subset $S \subseteq V(G)$,
if $L'$ is obtained from $L$ by restricting its domain to $S$, and if $\phi$ is an $(H[\bigcup_{v \in S} L(v)], L')$-coloring of $G[S]$, then we often simply say that $\phi$ is an $(H,L)$-coloring of $G[S]$.
Similarly, if $L^*(v) \subseteq L(v)$ for each $v \in S$ and $\phi(v) \in L^*(v)$ for each $v \in S$,
then we often say that $\phi$ is an $(H,L^*)$-coloring of $G[S]$.
We also write $H[S] = H[\bigcup_{v \in S} L(v)] $.

Given a subset $U\subseteq V(G)$, we write $\overline U = V(G) \setminus U$.
Given an integer $r \geq 1$, we write $rG$ for the multigraph obtained from $G$ by replacing each edge of $G$ with $r$ parallel edges.

\section{The potential function}
\subsection{Setup}
\label{sec:p-setup}
Our main approach is inductive with a \emph{potential function} similar to the one from \cite{BCKX}.
Let $G$ be a multigraph and $\ell:V(G) \rightarrow \mathbb N \cup \{0\}$.
An \emph{$\ell$-cover} of $G$ is a correspondence cover $(H,L)$ satisfying $|L(v)|= \ell(v)$ for each $v \in V(G)$.
We say that $(G,\ell)$ is \emph{minimal} if $G$ has an
$\ell$-cover
$(H,L)$ for which $G$ has no $(H,L)$-coloring, but for every proper subgraph $G'$ 
of $G$ and every $\ell$-cover $(H',L')$ of $G'$,
$G'$ has an $(H',L')$-coloring.
Note that when $\ell(v) = k-1$ for all $v \in V(G)$,
$(G,\ell)$ is minimal if and only if $G$ is correspondence $k$-critical.

Suppose a pair $(k,\lambda)$ is fixed and that $\ell(v) \leq k-1$ for each $v \in V(G)$.
We define a \emph{potential function} $\rho:V(G) \cup E(G) \rightarrow \mathbb Z$ as follows.
For each $v \in V(G)$, define 
\begin{equation}
\label{eqn:potential}
\rho(v) = 
\begin{cases}
    (k-1)\lambda+1 & \ell(v)=k-1 \\
    \ell(v)(\lambda+1)-(k-1) & \ell(v) \in [0,k-2].
\end{cases}
\end{equation}
We also write $F(v) = 1$ for each $v \in V(G)$ with $\ell(v) = k-1$ and $F(v) = 0$ for each $v \in V(G)$ with $\ell(v) \leq k-2$, so that $\rho(v) = \ell(v)(\lambda+1) - (k-1) + F(v)$ for each $v \in V(G)$.

Next, fix a total ordering on $E(G)$. Given an edge $e$ with 
endpoints $u,v$, if $e$ is the first edge in the total ordering that joins $u$ and $v$, then define $\rho(e) = -2\lambda$.
Otherwise, if $e$ is not the first edge joining $u$ and $v$, define $\rho(e) = -(2\lambda+\mu)$, where $\mu = 6$. 
We often use the symbol $\mu$ throughout the paper to clarify how certain expressions arise.
For each subset $S \subseteq V(G)$, we define 
\[\rho_{G,\ell}(S) = \sum_{v \in S} \rho(v) + \sum_{e \in E(G[S])} \rho(e).\]
We often omit $G$ or $\ell$ from the subscript when doing so creates no ambiguity.
We also often write $\rho(G) = \rho_{G,\ell}(V(G))$.

In order to prove Theorem \ref{thm:intro}, we prove the following stronger statement.
\begin{theorem}
\label{thm:main}
    Let $(k,\lambda)$ be a pair for which
    $k \geq 47$,
    $\lambda \in \{24,28\}$, and $k \geq 352$ whenever $\lambda=24$.
    Let $G$ be a multigraph, and let $\ell:V(G) \rightarrow \{0,1,\dots,k-1\}$.
    If $(G,\ell)$ is minimal, then $\rho(G) \leq -k+1$ or $G$ is a $K_k$.
\end{theorem}
Theorem \ref{thm:intro} follows from Theorem \ref{thm:main}
by letting $G$ be simple and letting $\ell(v)=k-1$
for each $v \in V(G)$.

We note that 
our Theorem \ref{thm:main}
has a similar form to Theorem 2.1 of \cite{BCKX}, which is essentially a reformulation of Theorem \ref{thm:BCKX} using a potential function.
However,
our upper bound of $-k+1$
is significantly less than the upper bound of $-2$ in Theorem 2.1 of \cite{BCKX}.
Our value $-k+1$
gives a much stronger induction hypothesis
that leads to arguments which are incompatible with the inductive framework of \cite{BCKX}.
In particular, our Lemma \ref{lem:rho-LB}
shows that
when $(G,\ell)$ is a minimum counterexample,
$\rho(S) > \frac k4 (\lambda-2)$ for every nonempty $S \subsetneq V(G)$,
which is a powerful fact with no analogue in \cite{BCKX}.

We also note that our potential function is similar to that of \cite{BCKX}, with a few differences. 
First, 
the potential function of \cite{BCKX}
assigns $\rho(v) = \ell(v) \lambda - 1$ when $\ell(v) \in [2,k-2]$
and $\rho(v) = \ell(v) \lambda - 2$ when $\ell(v) \in [0,1]$;
thus, our values $\rho(v)$ are lower than those of \cite{BCKX} whenever $\ell(v) \leq k-3$.
By assigning a lower potential to vertices $v$ with smaller list sizes $\ell(v)$,
we prevent edge cases
which would cause the stronger upper bound in Theorem \ref{thm:main} to be false.
For instance, if $G$ is a $K_2$ and $\ell(v)=1$ for each $v \in V(G)$, then $(G,\ell)$ is minimal.
In our current framework with reduced potentials, we have $\rho(G) = 2(\lambda - k + 2) - 2\lambda < -k+1$.
However, in the framework of \cite{BCKX},
we have $\rho(v) = \lambda-2$ for each $v \in V(G)$, so that $\rho(G) = -4$.
Thus, these reduced potentials are necessary to allow the bound of $-k+1$ in Theorem \ref{thm:main},
which in turn allows stronger inductive tools as discussed above.
Another difference is that while \cite{BCKX} assigns $\rho(e) = -(2\lambda+1)$ to edges $e$ that are not the ``first'' in a set of parallel edges, we assign $\rho(e)=-(2\lambda+6)$.
This difference is mainly to allow for certain computations to work more smoothly, avoid case analysis, and also to ensure that the dipole $(k-1)K_2$ does not violate Theorem \ref{thm:main} when paired with a function that assigns $\ell(v) = k-1$ to each vertex $v$.

\subsection{Properties of a minimum counterexample}
In order to prove Theorem \ref{thm:main} for an appropriate pair $(k,\lambda)$,
we 
consider a counterexample  $(G,\ell)$ for which $|V(G)| $ is minimized, and subject to this, $|E(G)|$ is minimized.
As $(G,\ell)$ is minimal, $G$ is connected.
We fix an $\ell$-cover $(H,L)$ of $G$ for which $G$ has no $(H,L)$-coloring.
We also fix a value 
$\tau = \frac 18 (\lambda-2) - \frac 12$.
Thus,
$\tau = \frac 94$ when $\lambda = 24$, and $\tau = \frac{11}{4}$ when $\lambda = 28$.
The goal of this subsection is
to establish some initial properties of $(G,\ell)$, and in particular, 
to establish bounds on values $\rho(S)$ for certain subsets $S \subseteq V(G)$.

\begin{lemma}
\label{lem:Kk-free}
$G$ has no $K_k$ subgraph.
\end{lemma}
\begin{proof}
    If $G = K_k$, then $G$ is not a counterexample to Theorem \ref{thm:main}.
    Otherwise, suppose that $G$ has a proper $K_k$ subgraph, which we call $K$.
    As $\ell(v) \leq k-1$ for each $v \in K$, there exists an $\ell$-cover $(H',L')$ of $K$
    for which $K$ has no $(H',L')$-coloring, contradicting the minimality of $(G,\ell)$.
\end{proof}

As $(G,\ell)$ is minimal, it follows that for every proper subset $S \subsetneq V(G)$, 
$G[S]$ has an $(H,L)$-coloring. 
Given an $(H,L)$-coloring $\phi$ of $G[S]$,
for each $v \in \overline S$,
we define $L_{\phi}(v) = L(v) \setminus N_H(\phi)$ and $\ell_{\phi}(v) = |L_{\phi}(v)|$.
We make the following observation.

\begin{lemma}
\label{lem:partial}
    Let $S \subsetneq V(G)$, and let $\phi$ be an $(H,L)$-coloring of $G[S]$.
    Then, $G[\overline S]$ has no $(H,L_{\phi})$-coloring.
\end{lemma}
\begin{proof}
    Suppose that $G[\overline S]$ has an $(H,L_{\phi})$-coloring $\psi$.
    Then, the union $\phi \cup \psi$ is an $(H,L)$-coloring of $G$, contradicting our choice of $(G,H,L)$.
\end{proof}

Lemma \ref{lem:partial} has the following immediate corollary.

\begin{lemma}
\label{lem:dl}
    For every $v \in V(G)$, $d(v) \geq \ell(v)$.
\end{lemma}

Our next lemma is borrowed from 
\cite{BCKX}
and shows that for each proper subset $S \subsetneq V(G)$, 
$\rho(S)$ has a lower bound that is proportional to the number of edges joining $S$ with its complement.
The original version of this lemma (\cite[Lemma 4.5]{BCKX})
is proven for a slightly different potential framework, so we include a proof for completeness.
\begin{lemma}
\label{lem:gap}
    If 
    $S \subsetneq V(G)$ is nonempty, then $\rho(S) \geq (\lambda-2)\|S,\overline S\|+1$.
\end{lemma}
\begin{proof}
    Suppose that the lemma is false, and let $S$ be the largest set for which the lemma does not hold.
    As $S$ is a counterexample to the lemma, $\rho(S) \leq (\lambda-2) \|S,\overline S\|$.
    
    Let $G' = G - S$.
    As $S \subsetneq V(G)$, it follows from the minimality of $(G,\ell)$ that $G[S]$ has an $(H,L)$-coloring $\phi$.
    Then, by Lemma \ref{lem:partial},
    $G'$ has no $(H,L_{\phi})$-coloring.
    Thus,
    there is a subset $U \subseteq V(G')$ and a spanning subgraph $G'' \subseteq G[U]$ such that
    $(G'',\ell_{\phi})$ is minimal.
    %
    As $(G,\ell)$ is a minimum counterexample,
    and as $G$ has no $K_k$ subgraph by Lemma \ref{lem:Kk-free},
    it follows that
    $\rho_{\ell_{\phi}}(U) \leq -k+1$.

Now, consider the set $S':=U \cup S$, and write $q = \|S,U\|$.
As $(G,\ell)$ is minimal and $U \subsetneq V(G)$,
it follows that $\ell$ and $\ell_{\phi}$ are not identical on $U$;
therefore, $q \geq 1$. 
 As $\rho_{\ell}(v)-\rho_{\ell_{\phi}}(v)\leq (\lambda+2)\|v, S\|$ for each $v\in U$, 
 it follows that
$\rho_{G,\ell}(U) \leq -k+1 + q (\lambda + 2)$.
Therefore,
\begin{equation}
\label{eqn:j(k-2)}
\rho_{\ell}(S') \leq 
\rho_{\ell}(U) + \rho_{\ell}(S) - 2\lambda q \leq -k+1 + q(\lambda + 2) - q(2\lambda) + \rho_{\ell}(S) = -k+1 - q(\lambda - 2) + \rho_{\ell}(S).
\end{equation}


    Now, write $\|S,\overline S\| = j \geq 1$. Then,
    as $\rho_{\ell}(S) \leq j(\lambda -2)$,
    (\ref{eqn:j(k-2)}) implies that 
    \begin{equation}
    \label{eqn:S'}
    \rho_{\ell}(S')     \leq -k+1  - q (\lambda -2) + j(\lambda -2) 
    =-k+1 + (j - q)(\lambda -2).
    \end{equation}
    If $S' = V(G)$, then $q=j$, so~\eqref{eqn:S'} implies that $\rho(G) \leq -k+1$, and $G$ is not a counterexample to Theorem~\ref{thm:main}.
    Therefore, $\|S', \overline{S'}\| \geq 1$, so the 
    maximality of $|S|$ implies that  $\rho_{\ell}(S') \geq \lambda -1$. Hence, $q< j$ by~\eqref{eqn:S'}. 
    Since $\|S,U\| = q$ and $\|S, \overline S\| \geq j$, it follows that $\|S', \overline{S'}\| \geq \|S, V(G') \setminus U\| \geq j - q \geq 1$.
    Then, the maximality of $|S|$
     tells us that $\rho_{\ell}(S') > (j-q)(\lambda -2)$, contradicting (\ref{eqn:S'}). 
\end{proof}

Lemma \ref{lem:gap} has the following corollary.

\begin{lemma}
\label{lem:deg-UB}
    Each vertex $v \in V(G)$ satisfies $d(v) \leq \frac{\lambda}{\lambda-2}\ell(v)$.
\end{lemma}
\begin{proof}
    Consider a vertex $v \in V(G)$.
    By Lemma \ref{lem:gap} and the definition of $\rho(v)$,    \[
        d(v)(\lambda-2)+1 \leq \rho(v) \leq 1 + \ell(v) (\lambda+1)-(k-1).
    \]
    Rearranging,
    \[
        d(v) \leq  \ell(v) \frac{\lambda+1}{\lambda-2} - \frac{k-1}{\lambda-2} \leq \frac{\lambda}{\lambda-2} \ell(v).
    \]
\end{proof}

The next lemma shows that for each proper subset $S\subsetneq V(G)$,
$\rho(S)$ has an upper bound that is proportional to the number of edges joining $S$ with its complement.
This lemma is similar to Lemma 4.16 from \cite{BKX}.

\begin{lemma}
\label{lem:bd-UB}
For each nonempty set $S \subsetneq V(G)$, 
$\rho(S) \leq (\lambda+2)\|S,\overline S\|-k+1$.
\end{lemma}
\begin{proof}
    Let $S$ be a minimum counterexample to the lemma, and write $j = \|S,\overline S\|$.
    If $G[S]$ is disconnected, then $S$ has two proper subsets $S_1$ and $S_2$ for which $\rho(S) = \rho(S_1)+ \rho(S_2)$. 
    By the minimality of $S$,
    $\rho(S_i) \leq (\lambda+2)\|S_i,\overline S\| - k+1$ for each $i \in \{1,2\}$, so that 
    $\rho(S) \leq j(\lambda+2)-2(k-1)$, contradicting the assumption that $S$ is a counterexample. Therefore, we assume that $G[S]$ is connected.

    By the minimality of $(G,\ell)$, $G[\overline S]$ has an $(H,L)$-coloring $\phi$. Since $G[S]$ has no $(H,L_{\phi})$-coloring, there is a subset $U \subseteq S$, containing at least one vertex with a neighbor in $\overline S$, for which $\rho_{\ell_{\phi}}(U) \leq -k+1$.
    Let $r = \|U,\overline S\|$.
    Then,
    $\rho_{\ell}(U) \leq -k+1 + r(\lambda+2)$.
    As $S$ is a counterexample and $r \leq j$, $U \subsetneq S$.

    Now, write $U' = S \setminus U$. Let $q = \|U,S \setminus U\|$, and observe that $\|U',\overline U'\| = q + j-r$.
    As $G[S]$ is connected, $q \geq 1$. 
    By the minimality of $S$, we know that $\rho(U') \leq (\lambda+2)(q+j-r)-(k-1)$. Therefore,
    \[
        \rho(S) \leq \rho(U') + \rho(U) -2q\lambda \leq (\lambda+2)(q+j-r)-(k-1) + r(\lambda+2) - (k-1) - 2q\lambda.
    \]
    As $\lambda \geq 2$, this upper bound 
    is strictly less than $j(\lambda+2)-(k-1)$, contradicting the assumption that $S$ is a counterexample.
\end{proof}

By combining Lemmas \ref{lem:gap} and \ref{lem:bd-UB},
we obtain the following universal lower bound for the potential of a nonempty proper subset of $V(G)$.
This lemma is perhaps the most important consequence of using an upper bound of $-k+1$ rather than $-2$ in Theorem \ref{thm:main}.

\begin{lemma}
\label{lem:rho-LB}
    For every nonempty proper subset $S \subsetneq V(G)$, 
    $\|S,\overline S\| \geq \frac k4$, and 
    $\rho(S) \geq \frac {k}{4} (\lambda-2) + 1$.
\end{lemma}
\begin{proof}
    Let $j = \|S,\overline S\|$. 
    By Lemmas \ref{lem:gap} and \ref{lem:bd-UB},
    \[
        j(\lambda-2) + 1 \leq \rho(S) \leq j(\lambda+2)-k+1.
    \]
    Combining the lower and upper bounds, $j \geq \frac k4$.
    Thus, by Lemma \ref{lem:gap},
    $\rho(S) \geq \frac k4 (\lambda-2)+1$.
\end{proof}




Finally, we show that no proper subset of $V(G)$ 
induces a subgraph that is too close to $K_k$.

\begin{lemma}
\label{lem:no-exceptions}
    For each $k$-vertex set $K \subsetneq V(G)$,
    $G[K]$ has fewer than $\binom k2 - \frac{\tau k}{\lambda} $ edges.
\end{lemma}
\begin{proof}
    Suppose that $K \subsetneq V(G)$ is a set of $k$ vertices inducing at least $\binom k2 - \frac{\tau k}{\lambda}$ edges.
    We aim to show that $\rho(K)$ is too small and hence violates Lemma \ref{lem:rho-LB}.
    We first
    note that a multigraph $G_0$ 
    on $k$ vertices
    with a function $\ell_0:V(G_0) \rightarrow \{0,\dots,k-1\}$
    and $\binom k2$ edges
    satisfies $\rho_{G_0,\ell_0}(V(G_0)) \leq k$. Therefore,
    as $\rho_{\ell}(K)$ is maximized when $(G[K],\ell)$ is obtained from 
    such a pair $(G_0,\ell_0)$ by
    removing at most $\frac{\tau k}{\lambda}$ edges,
    the fact that $\tau  = \frac 18(\lambda-2) - \frac 12$ implies that
     \[\rho(K) \leq  k + 2\tau k = \frac k4 (\lambda-2) < \frac k4(\lambda-2) + 1.\]
As $K \subsetneq V(G)$, we have a contradiction by Lemma \ref{lem:rho-LB}.
\end{proof}

\section{Low cliques}
The main goal of this section is to show that if $G$ has a large clique $K$, then at least one vertex $v \in V(K)$ satisfies $d(v) > \ell(v)$.
As we consider correspondence colorings of cliques, we need the following lemma, which is a special case of Theorem 5 from \cite{KimO}.

\begin{lemma}
\label{lem:clique-cover}
    Let $G_0$ be a $K_t$, and suppose that $(H_0,L_0)$ is a correspondence cover of $G_0$ satisfying $|L_0(v)| = t-1$ for each $v \in V(G_0)$. 
    If $G_0$ has no $(H_0,L_0)$-coloring, then $H_0$ has $t-1$ components isomorphic to $K_t$.
\end{lemma}

We say that a vertex $v \in V(G)$ is \emph{low} if $\ell(v) = d(v)$.
Let $\mathcal L$ be the subgraph of $G$ induced by low vertices.
We will ultimately show that no clique of $\mathcal L$ is too large.

\begin{lemma}
    \label{lem:L-is-GDP}
    $\mathcal L$ is a GDP-forest.
\end{lemma}
\begin{proof}
    If $\mathcal L$ has no vertex, then the lemma holds. Otherwise, suppose that $\mathcal L$ has at least one vertex, and let $T$ be a component of $\mathcal L$.
    By the minimality of $(G,\ell)$, the graph $G' = G-T$ has an $(H,L)$-coloring $\phi$.
    By Lemma \ref{lem:partial}, $T$ has no $(H,L_{\phi})$-coloring.
    Furthermore, as each $v \in V(T)$ is low, $\ell_{\phi}(v) \geq d_{\mathcal L}(v)$ for each $v \in V(T)$.
    Therefore, $(H,L_{\phi})$ is a degree cover of $T$ that admits no correspondence coloring, so by Theorem \ref{thm:GDP-char}, $T$ is GDP-tree. As $T$ was chosen arbitrarily, $\mathcal L$ is a GDP-forest.
\end{proof}

\begin{lemma}
    $\mathcal L$ has no $K_{\lceil(1-\frac{\tau}{\lambda})k \rceil }$ subgraph.
    
\end{lemma}
\begin{proof}
Suppose that
$t \geq (1-\frac{\tau}{\lambda}) k$
and
$K \subseteq V(\mathcal L)$ 
is a set of $t$ vertices
for which $G[K]$
has a spanning $K_t$ subgraph.
Choose $K$ so that $t$ is maximized.
As $G$ has no $K_k$  by Lemma \ref{lem:Kk-free}, $t \leq k-1$.
\begin{claim}
    $K \subsetneq V(G)$.
\end{claim}
\begin{proof}
If $V(G) = K$, then $G = \mathcal L$.
Then, as $\mathcal L$ is a GDP-tree by Lemma \ref{lem:L-is-GDP},
$G=rK_t$ for some $r \geq 1$, and $\ell(v) = d(v) = r(t-1)$ for each $v \in V(G)$.
If $r = 1$, then
$\ell(v) = d(v) \leq t-1 \leq k-2$ for each $v \in V(G)$, so
 \[
\rho(G) \leq t ((t-1)(\lambda+1) - (k-1)) - 2\lambda \binom t2. 
\]
The quadratic term of this function is $t^2$ with a coefficient of $1$, so the function above is maximized when $t$ is as small or as large as possible. If $t\in \{1,k-1\}$, then $\rho(G) = -k+1$.
As $1 \leq t \leq k-1$, we have
$\rho(G) \leq -k+1$, and 
$G$ is not a counterexample.
If $r \geq 2$, 
then
\[
\rho(G) \leq t (1+ r(t-1)(\lambda+1) - (k-1) ) - 2\lambda r \binom t2 - (r-1) \mu \binom t2.
\]
As the terms with $\lambda$ cancel,
we have 
\[
\rho(G) \leq  t(2-k+ r(t-1) ) - (r-1) \mu \binom t2 = t \left ( 2 - k + (t-1) \left (r - \frac 12 \mu (r-1) \right ) \right ).
\]
As $\frac 12 \mu (r-1) \geq r$, it follows that 
\[
\rho(G) \leq t(2-k) < -k+1,
\]
and again, $G$ is not a counterexample.
Therefore, $K \subsetneq V(G)$.
\end{proof}

\begin{claim}
\label{claim:Kt}
    $G[K]=K_t$.
\end{claim}
\begin{proof}
If not, then as $\mathcal L$ is a GDP-forest,
$G[K]=rK_t$ for some $r \geq 2$.
Then, for each $v \in K$,
$d(v) \geq r(t-1) \geq 2(t-1) > k-1 \geq \ell(v)$, contradicting the fact that $v$ is low. 
\end{proof}

For each $v \in K$, write $B(v) = N(v) \setminus K$.
\begin{claim}
\label{claim:B}
The sets of the family $\{B(v): v \in K\}$ are not all identical.
\end{claim}
\begin{proof}
Suppose that $B(u) = B(v)$ for every $u,v \in K$.
Write $B = B(v)$ for each $v \in K$.
We note that for every $w \in V(G) \setminus K$ with a neighbor in $K$, $K \subseteq N(w)$.

We consider two cases.

\vspace{0.5cm}

\noindent
\textbf{Case 1:}
$G[B \cup K]$ has a spanning clique.
Write $K' = B \cup K$ and $|B| = q$. Observe that as $G$ has no $K_k$ and $G[K']$ has a spanning clique, $t+q \leq k-1$.
For each $u \in K'$,
write 
\[\sigma(u) = \rho(u) - \lambda d(u) - \frac 12 \mu\sum_{u' \in N(u) \cap K'} (\|u,u'\| - 1),\] and observe that $\sum_{u \in K'} \sigma(u)=\rho(K') - \lambda \|K',\overline{K'}\|$.

\begin{claim}
\label{claim:K'-subset}
    $K' \subsetneq V(G)$.
\end{claim}
\begin{proof}
If $K' = V(G)$, then for each $v \in K$,
as $|N(v)| \leq t+q-1 \leq k-2$,
it follows that $\sum_{v' \in N(v)} (\|v,v'\| - 1) + (k-1-\ell(v)) \geq k-t-q$.
As each decrease in $\ell(v)$ decreases $\sigma(v)$ by at most $2$ while each parallel edge incident to $v$ decreases $\sigma(v)$ by $\frac 12 \mu > 2$, we conclude that $\sigma(v) \leq  - (k-t-q)$.
Furthermore, for each $w \in B$, $\rho(w) \leq 1-\lambda$.
Therefore,
\[
\rho(G) = \rho(K') \leq -t(k-t-q) + q (1-\lambda).
\]
As $q$ has a coefficient of $t - \lambda > \frac 9{10} k - \lambda > 0$ in the expression above, we thus set $q = k-t-1$ to maximize the expression and obtain
\[
\rho(G) \leq -t + (k-t-1)(1-\lambda).
\]
As the coefficient of $t$ in the expression above is $\lambda-2>0$, we thus use the maximum value $t=k-1$ to obtain $\rho(G) \leq -k+1$, a contradiction.
\end{proof}

By Claim \ref{claim:K'-subset} and Lemma \ref{lem:gap},
$\rho(K') > (\lambda-2)\|K',\overline{K'}\|$.
Therefore,
$\sum_{u \in K'} \sigma(u) > -2\|K',\overline{K'}\|$.

Now, for each $u \in K'$, write $r(u) = d(u) - d_{K'}(u)$, and observe that $r(v) = 0$ for each $ v\in K$.
By rearranging the inequality above, 
\begin{equation}
\label{eqn:r}
\sum_{ u\in K'} (\sigma(u) + 2r(u) ) > 0.
\end{equation}

Now, for each $w \in B$, we have 
$\rho(w) \leq 2-k + \ell(w) (\lambda+1) $ and
$r(w) \leq d(w) - (q+ t) + 1$.
Therefore,
\[
\sigma(w) + 2r(w) \leq \rho(w) - \lambda d(w) + 2r(w) \leq 4-k
+ \ell(w) (\lambda+1) -  (\lambda-2) d(w) - 2(q + t) .
\]
As
$d(w) \geq \ell(w) + 1$ by the maximality of $t$ and $\ell(w) \leq k-1$,
\begin{eqnarray}
\notag
\sigma(w) + 2r(w)&\leq& 4 - k - 2(q+t) + \ell(w)(\lambda+1) - (\lambda-2)(\ell(w) + 1) = 6 - k - 2(q+t) + 3\ell(w) \\
\label{eqn:w-UB}
&\leq& 2k +3 - 2(q+t).
\end{eqnarray}

Next, for each $v \in K$, we have $d(v) = \ell(v)$.
If $\ell(v) \leq k-2$, then $\rho(v) = \ell(v) (\lambda+1) - k + 1$.
Then, $\sigma(v) \leq \ell(v) (\lambda+1) - k + 1 - \lambda (q+t-1) - (\lambda + \frac 12 \mu) (\ell(v) - q - t + 1)$.
As this upper bound is decreasing as a function of $\ell(v)$, 
and as $\ell(v) = d(v) \geq q+t-1$,
we thus have 
\[
\sigma(v) \leq q+t-k.
\]
If $\ell(v) = k-1$, then $\ell(v) \geq q+t$, so a similar calculation shows that 
\[
\sigma(v) \leq \ell(v)(\lambda+1) -k+2 - \lambda(q+t-1) - (\lambda + \frac 12 \mu) (\ell(v) - q - t + 1) \leq q+t - k + 2 + \lambda - (\lambda + \frac 12 \mu) \leq q+t-k.
\]

Combining our upper bound for $\sigma(v)$ with \eqref{eqn:w-UB},
we have 
\[
\sum_{u \in K'} (\sigma(u) + 2r(u)) \leq 
t ((q+t) -k ) + q(2k-1 - 2(q+t)).
\]
As $q < \frac 14 k$ and $q+t < k$, we have $2k - t - 4q - 1> k - 3q - 1 >  0$, so
by considering the partial derivative of the expression above with respect to $q$,
this upper bound is increasing as a function of $q$ and hence is maximized when $q$ is as large as possible. As $q+t \leq k-1$, we thus set $q = k-t-1$,
which yields
\[
\sum_{u \in K'} ( \sigma(u) + 2r(u))  \leq -t + (k-t-1) = k-2t-1 < 0,
\]
contradicting \eqref{eqn:r}.


\vspace{0.5cm}

\noindent
\textbf{Case 2:}
There are two non-neighbors $w_1,w_2 \in B$.
Let $M \subseteq K$ be the set of vertices $v \in K$ for which $\|v,w_1\| + \|v,w_2\| > 2$, and write $m = |M|$.
As $K \subseteq N(w_1)$ and $K \subseteq N(w_2)$,
it follows that $m \leq d(w_1) + d(w_2)- 2t \leq \frac{2\lambda}{\lambda-2}(k-1)-2t$, 
where the last inequality follows from Lemma \ref{lem:deg-UB}.
As $t \geq (1 - \frac{\tau}{\lambda})k > \frac{9}{10}k$ and $\frac{2\lambda }{\lambda-2} < \frac{11}{5}$,
it follows that 
$m <\frac 25 k$.
We let $K' = K \setminus M \cup \{w_1,w_2\}$.

Let $\phi$ be an $(H,L)$-coloring of $G-K'$.
%
We claim that $G[K']$ has an $(H,L_{\phi})$-coloring.
To this end, we extend $\phi$ one vertex at a time and update the lists $L_{\phi}$ after each new assignment.
We also update the 
graph $G_{\phi}$ induced by the uncolored vertices of $G$, as well as the
value $d_{\phi}(v)$ denoting the number of uncolored neighbors of a vertex $v \in K'$.
Note that by construction and Claim \ref{claim:Kt}, $G_{\phi}$ has no parallel edges.

Note that initially, $d_{\phi}(v) \geq \ell_{\phi}(v) = t-m+1$ for each $v \in K \setminus M$.
Furthermore, initially, $d_{\phi}(w_1) =  d_{\phi}(w_2) = t-m$,
and as $d(w_i) \leq (1 + \frac{2}{\lambda-2} ) \ell(w_i)$ for each $i \in \{1,2\}$,
it follows that $\ell_{\phi}(w_i) \geq d_{\phi}(w_i) - \frac{2}{\lambda-2} \ell(w_i) \geq t - m-\frac{2(k-1)}{\lambda-2} $.

Now,
while $\ell_{\phi}(w_1) < d_{\phi}(w_1)$,
we execute the following step:
\begin{quote}
    Choose a vertex $v \in V(G_{\phi}) \cap K$, and assign a color $\phi(v) \in L_{\phi}(v)$ to $v$ with no neighbor in $L(w_1)$.
\end{quote}
Writing $q = |V(G_{\phi}) \cap (K\setminus M)|$,
we observe that each time the step is executed, $\ell_{\phi}(v) \geq d_{\phi}(v) = q  +1 > q = d_{\phi}(w_1) > \ell_{\phi}(w_1)$,
so an appropriate color $\phi(v)$ 
can be chosen.
Furthermore, each time the step is executed, $d_{\phi}(w_1)$ decreases by $1$, while $\ell_{\phi}(w_1)$ remains constant.
As we initially have $d_{\phi}(w_1) - \ell_{\phi}(w_1) \leq \frac{2}{\lambda-2}\ell(w_1) \leq \frac{2(k-1)}{\lambda-2}$,
the step is executed at most $\frac{2(k-1)}{\lambda-2}$ times.

Next, while $\ell_{\phi}(w_2) < d_{\phi}(w_2)$, we execute the following step:
\begin{quote}
    Choose a vertex $v \in V(G_{\phi}) \cap K$, and assign a color $\phi(v) \in L_{\phi}(v)$ to $v$ with no neighbor in $L(w_2)$.
\end{quote}
Writing $q = |V(G_{\phi}) \cap (K\setminus M)|$,
we observe that each time the step is executed, $\ell_{\phi}(v) \geq d_{\phi}(v) =  q+1> q = d_{\phi}(w_2) > \ell_{\phi}(w_2)$,
so an appropriate color $\phi(v)$ 
can be chosen.
As above, we execute the step at most $\frac{2(k-1)}{\lambda-2}$ times.

Therefore, after assigning colors 
to
at most $\frac{4(k-1)}{\lambda-2} < \frac 15 k < t- m-3$ vertices of $K\setminus M$,
we obtain a partial coloring $\phi$
such that $\ell_{\phi}(v) \geq d_{\phi}(v)$ for each $v \in V(G_{\phi})$.
As $V(G_{\phi}) \cap (K\setminus M)$ contains a triangle, it follows that $G_{\phi}$ has an induced $K_5^-$ (obtained from a $K_5$ by deleting one edge); 
therefore, by Theorem \ref{thm:GDP-char},
$G_{\phi}$ has an $(H,L_{\phi})$-coloring, and hence we can extend $\phi$ to an $(H,L)$-coloring of $G$.
This final contradiction proves the claim.
\end{proof}

By the minimality of $(G,\ell)$, there is an $(H,L)$-coloring $\phi$ of $G-K$.
As $d_K(v) = t-1$ for each $t \in K$, $\ell_{\phi}(v) \geq t-1$ for each $v \in K$.
As $G[K]$ has no $(H,L_{\phi})$-coloring by Lemma \ref{lem:partial},
Lemma
\ref{lem:clique-cover} implies that 
$H[\bigcup_{v \in K} L_{\phi}(v) ] $
consists of $t-1$ copies of $K_t$. 
We consider two cases.

\vspace{0.5cm}

\noindent
\textbf{Case 1:} There is a value $s \geq t$ such that $H[K]$ consists of 
$s$ copies of $K_t$.
In this case,
for each $v \in K$,
we have
$\ell(v) =s$,
and $v$ is joined to a set $B(v):=N(v) \setminus K$
by exactly $s-t+1$ edges.
By Claim \ref{claim:B},
there are $v,w \in K$
such that there is some vertex 
$x \in B(v) \setminus B(w)$.

Now, label the edges joining $x$ to $B(w)$ as $e_1,\dots e_{s-t+1}$, and write $x_i$ for the endpoint of $e_i$ in $B(w)$.
Also, 
let $M^*_x$ be a maximal matching in $H$ joining $L(x)$ with $L(v)$, and 
for each $i \in \{1, \dots, s-t+1\}$,
let 
$M^*_i$ be a matching in $H$ joining $x_i$ with $w$.
Also, for each $i \in \{1,\dots,s-t+1\}$,
let
$M_i$ be a matching in the complete bipartite graph between $L(x)$ and $L(x_i)$ such that a node $\alpha \in L(x)$ is joined to a node $\beta \in L(x_i)$ if and only if $\alpha$ and $\beta$ are adjacent to the same component of $H[K]$ via the respective matchings $M^*_x$ and $M^*_i$.
Then let $H' = H + \bigcup_{i=1}^{s-t+1} M_i$.
We write $G' = G-K+\{xx_1, \dots , xx_{s-t+1}\}$.
We observe that $G'$ has no $K_k$ subgraph; indeed, if some $k$-vertex 
set $K' \subseteq V(G')$ induces at least $\binom k2$ edges in $G'$,
then $G[K']$ contains at least $\binom k2 - (s-t+1) \geq \binom k2 - (k-t) \geq \binom k2 - \frac {\tau  k}{\lambda}$ edges, contradicting Lemma \ref{lem:no-exceptions}.

We claim that $G'$ has no $(H',L)$-coloring. 
Indeed, suppose that $G'$ has an $(H',L)$-coloring $\phi$.
By Lemma \ref{lem:clique-cover},
$(H[K],L_{\phi})$
consists of $t-1$ copies of $K_t$.
For each $v \in K$, 
write $X(v) = L(v) \setminus L_{\phi}(v)$, and note that for each $c \in X(v)$, $c$ has a neighbor in $\phi$.
Furthermore, $(H[K],X)$ 
consists of $s-t+1$
$K_t$ components.
Now, let $c \in X(v)$
be the color in $X(v)$ adjacent to $\phi(x)$, and let $c' \in X(w)$ be the color in $X(w)$ belonging to a common component with $c$ in $(H,X)$.
Because $c' \in X(w)$, without loss of generality, $c'$ is adjacent to $\phi(x_1)$.
Since $\phi(x)$ and $\phi(x_1)$ are adjacent to the same component of $(H[K],L)$,
there is an edge $\phi(x) \phi(x_1) \in E(H')$.
This contradicts the assumption that $\phi$ is an $(H',L)$-coloring of $G'$.
Therefore, $G'$ has no $(H',L)$-coloring.

Therefore, there is a subset $U \subseteq V(G')$ with $\rho_{G'}(U) \leq -k+1$.
Therefore, 
\[\rho_G(U) \leq -k+1 + (2\lambda+\mu)(s-t+1) \leq (2\lambda+\mu) \frac{ \tau  k}{\lambda}  <  \frac 14 k (\lambda-2).\]
As $U \subsetneq V(G)$, this contradicts
Lemma \ref{lem:rho-LB}.

\vspace{0.5cm}

\noindent
\textbf{Case 2:} Some component $J$ of $H[K]$ is not a $K_t$. Choose $v \in K$ such that $L(v)$ 
contains some node $\alpha \in J$.
By Lemma \ref{lem:clique-cover},
for every $(H,L)$-coloring $\phi$
of $G-K$, $H[\bigcup_{v \in K} L_{\phi}(v) ] $
consists of $t-1$ components isomorphic to $K_t$;
therefore, $\phi$ contains some neighbor of $\alpha$.
Define $L'(w) = L(w) \setminus N_H(\alpha)$ for each $w \in V(G) \setminus K$.
Then, $G-K$ has no $(H,L')$-coloring.
As $\|v,\overline K\| \leq \frac{\tau k}{\lambda}$,
$L'$ is obtained from $L$ by deleting a total of at most 
$\frac{\tau  k}{\lambda}$ nodes.

As $G-K$ has no $(H,L')$-coloring, 
there is
$U \subseteq V(G-K)$ such that $\rho_{\ell'}(U) \leq -k+1$, so that $\rho_{\ell}(U) \leq -k+1+ \frac{\tau  k}{\lambda}(\lambda+2) < \frac 14 (\lambda-2) k$.
As $U \subsetneq V(G)$, 
this contradicts Lemma \ref{lem:rho-LB}.

Both cases lead to a contradiction, so the proof is complete.
\end{proof}

\section{Discharging}
\subsection{A discharging lemma for GDP-trees}
In this section, we consider a GDP-tree $T \subseteq G$ of 
maximum degree at most $k-1$.
Our goal is to establish a lemma showing that 
during our upcoming discharging argument, every GDP-tree in $G$ receives enough negative charge to complete the discharging argument successfully.
The function $\Phi(T)$ that we define below will ultimately count the amount of negative charge that a GDP-tree $T$ holds after (R2) in our upcoming discharging procedure.
The lemma that we establish in this section is similar to Lemma 3.2 of \cite{BCKX}.

We write
$\alpha = \frac{\lambda-2}{k}$.
For each $v \in V(T)$, 
define $m_T(v) = \frac 13 \mu \sum_{w \in N_T(v)} \left ( |E_T(v,w)| - 1 \right )$.
For each $v \in V(T)$, let
\begin{equation}
\label{eqn:Phi}
\Phi_T(v) = -\rho(v) + \lambda \ell(v) + (\ell(v) - d_T(v))\alpha + m_T(v).
\end{equation}
As $\rho(v) \leq  \ell(v) (\lambda+1) - k+2$, it follows that
\[
\Phi_T(v) \geq  k-2-\ell(v) + (\ell(v)-d_T(v))\alpha + m_T(v) = k-2-\alpha d_T(v) + \ell(v) (\alpha -1) + m_T(v).
\]

Since this lower bound decreases as a function of $\ell(v)$, we use the upper bound $\ell(v) \leq k-1$ to observe that
\[
\Phi_T(v) \geq -1 +\alpha(k-1- d_T(v)) + m_T(v) \geq -1.
\]
Finally, we define 
\[
\Phi(T)  =  \sum_{v \in V(T)} \Phi_T(v).
\]
\begin{lemma}
\label{lem:GDP}
    Let $T$ be a GDP-tree of maximum degree at most $k-1$ and with no $K_t$ subgraph for $t \geq (1-\frac{\tau}{\lambda}) k$. Then, $\Phi(T) \geq |T|$.
\end{lemma}
\begin{proof}
    We proceed by induction on $|T|$.
    First, suppose $T$ is a single 
    vertex $v$. 
    Then $\Phi(v) \geq -1 + \alpha(k-1) = -1 + \frac{k-1}{k} (\lambda-2) > 1$.

    Now, suppose $|T| \geq 2$.
    Let $B$ be a terminal block of $T$. Let $x$ be the unique cut-vertex of $T$ in $B$ if $B \subsetneq V(T)$, and let $x \in V(B)$ be arbitrarily chosen if $B = T$.
    Let $T' = T - (B-x)$.
    By the induction hypothesis, $\Phi(T') \geq  |T'|$. We aim to show that $\Phi(T) - \Phi(T') \geq |T| - |T'|$.

    If $B = K_q$, then $q \leq  \lceil (1-\frac{\tau }{\lambda})k \rceil - 1$. Then for each $v \in V(B - x)$, we have $d_T(v) < (1-\frac{\tau}{\lambda}) k-1$, so that 
    \[\Phi_T(v) > -1 + \alpha \left (\frac{\tau  k}{\lambda} \right ) = -1 + \frac{\lambda-2}{k} \left (\frac{\tau  k}{\lambda} \right ) 
    = -1 + \tau - \frac{2 \tau}{\lambda} .\]
    Also, $\Phi_T(x)-\Phi_{T'}(x) = -(q-1) \alpha $.
    Therefore, 
    \[\Phi(T) - \Phi(T') \geq (q-1)\left ( -1 + \tau - \frac{2  \tau}{\lambda} \right ) - \alpha(q-1) = \left (  -1 + \tau - \frac{2  \tau}{\lambda} - \frac{\lambda-2}{k} \right )(q-1) \geq q-1=|B-x|,\]
    and the lemma holds by induction.
    (The lower bounds for $k$ in Theorem \ref{thm:main} are chosen so that this inequality holds.)


    If $B = rK_q$ for some $r \geq 2$, then for each $v \in V(B-x)$, $d_T(v) = (q-1)r$, and so
    $\Phi_T(v) \geq -1 + \alpha (k-1-(q-1)r) + \frac 13 \mu (q-1)(r-1)$. We write $d = (q-1)r$ so that for each $v \in V(B-x)$,
\begin{eqnarray*}
\Phi_T(v) &\geq& -1 + \alpha (k-1-d) + \frac 13 \mu d\left (1-\frac 1r \right ) = -1 + \alpha(k-1) + d\left (-\alpha + \frac 13 \mu \left (1-\frac 1r \right ) \right ) \\
&> & -1 + \alpha(k-1) = -1 + \frac{k-1}{k} (\lambda-2) \geq 1.
\end{eqnarray*}

    Also, $\Phi_T(x)-\Phi_{T'}(x) = \frac 13 \mu (q-1)(r-1) - \alpha (q-1)r$. Since $r \geq 2$, we have $r-1 \geq \frac 12 r$, so
    $\Phi_T(x)-\Phi_{T'}(x) \geq \frac 16 \mu (q-1)r - \alpha (q-1)r = (\frac 16 \mu - \alpha) d > 0$.
    Thus, $\Phi(T)-\Phi(T') > |B-x|$, and we are done by induction.

    If $B = C_q$ for some $q \geq 4$, then for each $v \in V(B-x)$,
    \[\Phi_T(v) \geq -1 + \alpha(k-3) = -1 + \frac{k-3}{k} (\lambda-2) .\] 
    Also $\Phi_T(x)-\Phi_{T'}(x) = -2\alpha$ so 
    $\Phi(T) - \Phi(T') \geq \left ( -1 + \frac{k-3}{k} (\lambda-2) \right )|B-x| - 2\alpha > |B-x|$, and
    the lemma holds
    by induction.

    If $B = rC_q$ for some $q \geq 4$ and $r \geq 2$, then $\Phi_T(x)-\Phi_{T'}(x) = -2r \alpha + \frac 23 \mu(r-1) \geq -2r \alpha +\frac 13 \mu r > 0$.
    Also, for each $v \in V(B-x)$, we have 
    \[\Phi_T(v) \geq -1 + \alpha(k-1-2r) + \frac 23 \mu(r-1) \geq -1 + \frac{k-1}{k} (\lambda-2) + r(-2\alpha + \frac 13 \mu)  > 1.\]
    Thus $\Phi(T) - \Phi(T') >  |B-x|$, and the lemma holds by induction.

    As $B$ is a multiple of a clique or cycle, this exhausts all cases.
     %
%
\end{proof}

\subsection{The discharging argument}
To complete the proof of Theorem \ref{thm:main},
we use a discharging argument to show that $\rho(G) \leq -k+1$.
We give each vertex $v$ a charge $\sigma_1(v) = \rho(v) - \lambda d(v)$
and give each edge $e$ a charge of $\sigma_1(e) = \rho(e) + 2\lambda$,
so that $\rho(G) = \sum_{x \in V(G) \cup E(G)} \sigma_1(x)$.
Then, we execute the following discharging rules:
\begin{enumerate}
    \item[(R1)]
    For each $v \in V(G) \setminus V(\mathcal L)$ and edge $e$ incident to $v$, 
    let $v$ take charge $\frac{\lambda-2}{k}$ from
    the vertex joined to $v$ by $e$.
    Write $\sigma_2(v)$ for the resulting charge at $v$.
    \item[(R2)] For each $u \in V(\mathcal L)$, 
    after $u$ gives 
    charge according to (R1),
    let $u$ give charge $\frac 13 \mu$ to each edge $e \in E(\mathcal L)$ incident to $u$ satisfying $\rho(e) = -(2\lambda + \mu)$.
    Let the resulting charge of $u$
    be called $\sigma_2(u)$, and let the resulting charge at $e$ be called $\sigma_2(e)$.
\item[(R3)] For each $e \in E(G) \setminus E(\mathcal L)$ and each edge $e \in E(G)$ for which $\rho(e) = -2\lambda$, write $\sigma_2(e) = \sigma_1(e)$.
\end{enumerate}

Observe that 
(R1)-(R3) preserve overall charge, 
and $\sigma_2(e) \leq 0$ for each $e \in E(G)$.
In particular, if $\rho(e) = -(2\lambda + \mu)$, then $\sigma_2(e) \leq -\frac 13 \mu$.
Therefore,
$\rho(G) \leq \sum_{v \in V(G)} \sigma_2(v)$. We carry out an additional stage of discharging.

\begin{itemize}
    \item[(R4)] Write $\sigma_2(\mathcal L) = \sum_{v \in V(\mathcal L)} \sigma_2(v) $.
    If $\mathcal L$ is nonempty, then for each $v \in V(\mathcal L)$, let $\sigma_3(v) = \frac{\sigma_2(\mathcal L)}{|V(\mathcal L)|}$.
\item[(R5)] For each $v \in V(G) \setminus V(\mathcal L)$, let $\sigma_3(v) = \sigma_2(v)$.
\item[(R6)] For each edge $e \in E(G)$, write $\sigma_3(e) = \sigma_2(e)$.
\end{itemize}

Observe that
(R4) and (R5) preserve overall charge, so
$\rho(G) \leq \sum_{v \in V(G)} \sigma_3(v)$.

\begin{claim}
    If $v \in V(G) \setminus  V(\mathcal L)$, then $\sigma_3(v) \leq  -1$.
\end{claim}
\begin{proof}
If $v \in V(G) \setminus  V(\mathcal L)$, then 
\begin{eqnarray*}
\sigma_1(v) &=& \rho(v) - \lambda d(v) \\
&\leq& 1 + \ell(v) (\lambda+1) - (k-1) - \lambda d(v) \\
&=& 2 + \ell(v) - k + \lambda (\ell(v)-d(v)) 
\end{eqnarray*}
Therefore, by (R1),
\[
\sigma_3(v) = \sigma_2(v) \leq 2 + \ell(v) - k + \lambda(\ell(v)-d(v)) + \frac{\lambda-2}{k} d(v).
\]
Since $-\lambda + \frac{\lambda-2}{k} < 0$, 
it follows that this upper bound for $\sigma_3(v)$ is decreasing as a function of $d(v)$; therefore, as $d(v) \geq \ell(v)+1$,
\[
\sigma_3(v) \leq 2 + \ell(v) - k - \lambda + \frac{\lambda-2}{k} (\ell(v)+1) \leq 1 - \lambda + \lambda - 2 = -1.
\]
\end{proof}
\begin{claim}
\label{claim:L}
    For each $u \in V(\mathcal L)$,
    $\sigma_3(u) \leq -1$.
\end{claim}
\begin{proof}
    Let $u \in V(\mathcal L)$, and let $T$ be the component of $\mathcal L$ containing $u$.
    By Lemma \ref{lem:L-is-GDP}, $T$ is a GDP-tree.
    By (R1) and (R2),
    we have 
    \[\sigma_2(u) = \rho(u) - \lambda \ell(u) + (d_T(v) - \ell(v))\frac{\lambda-2}{k} - \frac 13 \mu \sum_{w \in N_T(u)} (|E_T(u,w)| - 1) =  -\Phi_T(u),\]
    where $\Phi_T(u)$ is defined in \eqref{eqn:Phi}.
    Therefore, 
    by Lemma \ref{lem:GDP},
    $\sum_{u \in V(\mathcal L)} \sigma_2(u) \leq -|V(\mathcal L)|$. Hence, $\sigma_3(u) = \frac{\sigma_2(\mathcal L)}{|V(\mathcal L)|} \leq -1$.
\end{proof}

Now,
let $v$ be a vertex for which $\ell(v)$ is largest. 
By Lemma \ref{lem:dl},
$d(v) \geq \ell(v)$.
First, suppose $\ell(v)=k-1$.
Since each $w \in V(G)$ satisfies
$\sigma_3(w) \leq -1$, 
it follows that 
\[\rho(G) \leq \sum_{x \in V(G) \cup E(G)} \sigma_3(x) \leq -(|N(v)|+1) - \frac 13 \mu (d(v) - |N(v)|) \leq  - (d(v) + 1) \leq -k.\]
This contradicts the choice of $(G,\ell)$.

Next, suppose $\ell(v)=k-2$.
Then, for each $w \in V(G)$, $\sigma_1(w) \leq -1$.
Then, the total charge in $G$ is at most 
\begin{eqnarray*}
\rho(G) &=& \sum_{x \in V(G) \cup E(G)} \sigma_1(x) \leq 
-(1 + |N(v)|) - \mu(d(v) - |N(v)|) \\
&\leq& 
-(1+d(v)) \leq -(1 + \ell(v)) = -k+1.
\end{eqnarray*}

Finally, suppose $\ell(v) \leq k-3$.
Then, $\sigma_1(w) \leq -2$ for each $w \in V(G)$.
Each edge $e$ incident to $v$ either carries a charge $\sigma_1(e) = -\mu < -2$ or joins $v$ to a vertex $w$ with $\sigma_1(w) \leq -2$.
Therefore, if $d(v) \geq \frac {k-1}{2}$, then
$\rho(G) = \sum_{w \in V(G)} \sigma_1(w) + \sum_{e \in E(G)} \sigma_1(e) \leq -k+1$, a contradiction.
Therefore, $\ell(v) \leq d(v) \leq \frac k2-1$.
Hence, for each $w \in V(G)$,
$\sigma_1(w) \leq -\frac k2 $.
If $|V(G)| \geq 2$, then 
$\rho(G) \leq \sum_{w \in V(G)} \sigma_1(w) \leq -k$, contradicting the choice of $(G,\ell)$.
Therefore, $|V(G)| = 1$, so that $\ell(v)=d(v)=0$, and hence $\rho(G) = -k+1$, contradicting the choice of $(G,\ell)$.
This final contradiction completes the proof.

\section{Density of graphs with minimal covers}
 
Let $G$ be a multigraph, 
and let $(H,L)$ be a correspondence cover of $G$.
We say that $(G,H,L)$ is \emph{minimal} 
if $G$ has no $(H,L)$-coloring, but
if $H'$ is obtained from $H$ by removing any edge,
$G$ has an $(H',L)$-coloring, and also $G-v$ has an $(H,L)$-coloring for each $v \in V(G)$.
In this section, we show that if 
$(G,H,L)$ is a minimal triple, then 
a similar result to Theorem \ref{thm:main} holds.
In particular, when $G$ is a simple graph, the bounds in Theorem \ref{thm:intro} also hold.
As a corollary, we prove
Theorem \ref{thm:fnk}.

We briefly note that the notion of a minimal triple $(G,H,L)$ is closely related
to the notion of a minimal unsatisfiable constraint satisfaction problem.
A (binary) \emph{constraint satisfaction problem (CSP)}
is a triple $(X,D,C)$, 
where $X = \{X_1, \dots, X_n\}$
is a set of variables,
$D = \{D_1, \dots, D_n\}$ is a set of corresponding variable domains, 
and $C = \{C_1, \dots, C_m\}$ is
a set of constraints.
Each constraint $C_j$ consists
of a pair of indices $\{s_j,t_j\} \subseteq \{1, \dots, n\}$, along with a binary relation $R_j$
that assigns a truth value to each element 
of
$D_{s_j} \times D_{t_j}$.
Then, an \emph{evaluation}
$\phi$ that maps $\phi(X_i) \in D_i$ for each $i$
is \emph{consistent}
if for each constraint $C_j \in C$,
the values $(\phi(X_{s_j}), \phi(X_{t_j}))$ form a pair that is true with respect to the relation $R_j$.
A CSP naturally gives rise to a \emph{primal graph} on $X$ with edges given by the pairs $t_j$ in $C$.
A  CSP $(X,D,C)$ is \emph{anti-functional}
if for each 
$C_j \in C$ and corresponding pair $\{s_j,t_j\}$,
if $x \in D_{s_j}$,
then there is at most one $y \in D_{t_j}$ for which $xy$ is false for the binary relation $R_j$ (and vice versa).
CSPs play a key role in the theoretical study of artificial intelligence; for further discussion, see e.g.~\cite{Foundations}.

The CSP $(X,D,C)$ is a \emph{minimal unsatisfiable CSP}
if $(X,D,C)$
has no consistent evaluation but has a consistent evaluation after removing any single constraint.
Thus, a minimal unsatisfiable anti-functional CSP $\mathcal X$ has an exact representation as a minimal triple $(G,H,L)$ whenever its primal graph has no isolated vertex.
The following theorem shows that
for each $d \geq 3$, if
a minimal unsatisfiable anti-functional
CSP
with size-$d$ domains
has a connected primal graph $G$,
then $G$ either contains a $K_{d+1}$ or has average degree at least $d+\frac 1{28}$.

\begin{theorem}
\label{thm:GHL-min}
    Let $(k,\lambda)$ be a pair for which $k \geq 4$, $\lambda \in \{24,28\}$, and $k \geq 352$ whenever $\lambda = 24$.
    Let $G$ be a simple $n$-vertex graph,
    and let $(G,H,L)$ be a minimal triple.
    If $\ell(v) = |L(v)| = k-1$ for each $v \in V(G)$,  then $|E(G)| > (k-1+\frac 1{\lambda}) \frac n2$ or $G$ has a $K_k$-subgraph.
\end{theorem}
\begin{proof}
    If $G$ has a $K_k$-subgraph, then we are done; therefore, we assume that $G$ is $K_k$-free.

    First, suppose that $k \geq 47$.
    We show that $\rho_{\ell}(V(G)) \leq -k+1$.
    As $(G,H,L)$ is minimal, $G$ has no $(H,L)$-coloring, and hence there exists a subgraph $G'$ of $G$ for which $(G',\ell)$ is minimal.
    As $G'$ is $K_k$-free, $\rho_{\ell}(V(G')) \leq -k+1$.
    
    We let $S \subseteq V(G)$ be the largest set for which $\rho(S) \leq -k+1$.
    As $|S| \geq |V(G')|$, $S$ is nonempty.
    If $S = V(G)$, then $\rho_{\ell}(V(G)) \leq -k+1$ by Theorem \ref{thm:main}.
    Otherwise,
    $S \subsetneq V(G)$.
    As $(G,H,L)$ is minimal,
    there is an $(H,L)$-coloring $\phi$ of $G[S]$.
    As $G$ has no $(H,L)$-coloring, the graph $G -S$ has no $(H, L_{\phi})$-coloring.
    Therefore, $G-S$ has some subgraph $G''$ such that $(G'',\ell_{\phi})$ is minimal.
    By Theorem \ref{thm:main},
    $\rho_{\ell_{\phi}}(V(G'')) \leq -k+1$.

    Now, let $S' = S \cup V(G'')$.
    Note that for each $v \in V(G'')$,
    $\ell_{\phi}(v) \geq k - 1-\|v,S\|$,
    so that $\rho_{\ell'_{\phi}}(v) \geq \rho(v) - (\lambda+2)\|v,S\|$.
    Therefore, $\rho(V(G'')) \leq \rho_{\ell_{\phi}}(V(G'')) + (\lambda+2)\|V(G''), S\|$, and hence 
    \[
        \rho(S') \leq \rho(S) + \rho(V(G'')) - 2\lambda \|V(G''), S\| \leq (-k+1) + (-k+1) - (\lambda-2)\|V(G''),S\| < -k+1,
    \]
    contradicting the maximality of $S$.
    Therefore, $\rho_{\ell}(V(G)) \leq -k+1$.
    Thus,
    it follows that
    $((k-1)\lambda+1)|V(G)| - 2\lambda |E(G)| \leq -k+1$.
    Rearranging, we find that $|E(G)| \geq (k-1+\frac 1{\lambda})\frac n2 + \frac{k-1}{2\lambda} > (k-1+\frac 1{\lambda}) \frac n2$,
    which is what we needed to show.

    Now, suppose that $4 \leq k \leq 46$.
    We repeat the same argument above with $\lambda = \left \lceil \frac{k^2-7}{2k-7} \right \rceil < 28$.
    Instead of Theorem \ref{thm:main}, we use Theorem 2.1 of \cite{BCKX}, which mutatis mutandis shows that $|E(G)| \geq (k-1+\frac{1}{\lambda}) \frac n2 + \frac{1}{\lambda} > (k-1+\frac{1}{28}) \frac n2$, which is what we needed to show.
\end{proof}

We observe that 
the first part of the argument in Theorem \ref{thm:GHL-min}
implies the following more technical theorem without any changes.

\begin{theorem}
    Let $(k,\lambda)$ be a pair for which $k \geq 47$, $\lambda \in \{24,28\}$, and $k \geq 352$ whenever $\lambda = 24$.
    Let $G$ be a multigraph,
    and let $(G,H,L)$ be a minimal triple.
    If $\ell(v) = |L(v)| \leq k-1$ for each $v \in V(G)$,  then $\rho_{\ell}(V(G)) \leq -k+1$ or $G$ has a $K_k$-subgraph.
\end{theorem}
For $5 \leq k \leq 46$, a theorem similar to Theorem \ref{thm:GHL-min}
can be obtained using
$\lambda=\left \lceil \frac{k^2-7}{2k-7} \right \rceil$ and
the definition of
$\rho$ from \cite{BCKX}. 
When $k=4$,
then the main result of \cite{k=4}
implies that under the same conditions, $|E(G)| > (3+\frac 15) \frac n2$ or $G$ has a subgraph isomorphic to $K_4$ or a Moser spindle, which is a graph on $7$ vertices with $11$ edges and hence has average degree $3+\frac 17$.

Now, we prove Theorem \ref{thm:fnk} as a corollary of Theorem \ref{thm:GHL-min}.
We restate Theorem \ref{thm:fnk}
for the reader's convenience.
\begin{theorem}
\label{thm:fnk-2}
    Let $k \geq 47$ and $n \geq k+2$ be integers. Then,
    \begin{enumerate}
        \item $f_{\ell}(n,k) \geq (k-1+\frac 1{28}) \frac n2 + \frac{k-1}{56}$, and $f_{DP}(n,k) \geq (k-1+\frac 1{28}) \frac n2 + \frac{k-1}{56}$;
        \item If $k \geq 352$,
        then $f_{\ell}(n,k) \geq (k-1+\frac 1{24}) \frac n2 + \frac{k-1}{48}$, and $f_{DP}(n,k) \geq (k-1+\frac 1{24}) \frac n2 + \frac{k-1}{48}$.
    \end{enumerate}
\end{theorem}
\begin{proof}
    The lower bounds for $f_{DP}(n,k)$ follow immediately from Theorem \ref{thm:main}. Therefore, our main task is to show that the lower bounds for $f_{\ell}(n,k)$ also hold.

    Let $k \geq 47$, and let $n \geq k+2$. Let $\lambda \in \{24,28\}$,
    with $\lambda=28$ whenever $k < 352$.
    Let $G$ be a list $k$-critical graph on $n$ vertices, and let $L$ be a list assignment on $G$ satisfying $|L(v)| = k-1$ for each $v \in V(G)$ such that $G$ has no $L$-coloring.
    Let $(H',L')$ be a correspondence cover of $G$ obtained as follows.
    For each $v \in V(G)$, let $L_0(v) = \{(v,c): c\in L(v)\}$.
    For each edge $uv \in E(G)$ and $c \in L(u) \cap L(v)$,
    add an edge $(u,c) (v,c)$ to $H'$.
    Observe that there is a bijection between $L$-colorings of $G$ and $(H',L')$-colorings of $G$.

    As $G$ has no $(H',L')$-coloring, there is a subgraph $H''$ of $H'$ for which $(G,H'',L')$ is minimal.
    By Theorem \ref{thm:GHL-min},
     as $|L(v)| = k-1$ for each $v \in V(G)$, it follows that 
    $\rho(G) \leq -k+1$.
    Therefore,
    $((k-1)\lambda+1)|V(G)| - 2\lambda |E(G)| \leq -k+1$.
    Rearranging, we find that $|E(G)| \geq (k-1+\frac 1{\lambda})\frac n2 + \frac{k-1}{2\lambda}$,
    which is what we needed to show.
\end{proof}

\section{Acknowledgment}
The main idea for this paper arose while working on a related paper with Alexandr Kostochka and Zimu Xiang \cite{BKX}. I am grateful to them for fruitful discussions.

\bibliographystyle{plain}
\bibliography{bib}

\end{document}